\documentclass[final,smallextended]{amsart}
\usepackage[latin9]{inputenc}
\usepackage{amstext}
\usepackage{amsthm}
\usepackage{amssymb}
\usepackage{url}

\makeatletter
\numberwithin{equation}{section}
\theoremstyle{plain}

  \theoremstyle{plain}
  
  \theoremstyle{definition}
  
  \theoremstyle{plain}

  \newenvironment{svmultproof}{\begin{proof}}{\end{proof}}

\@ifundefined{date}{}{\date{}}

\usepackage{tikz}

\input{tcilatex}

\AtBeginDocument{
  
}

\makeatother

  \providecommand{\corollaryname}{Corollary}
  \providecommand{\definitionname}{Definition}
  
  \providecommand{\propositionname}{Proposition}
  \providecommand{\theoremname}{Theorem}

\begin{document}

\title{Supersymmetry and cohomology of graph complexes.}
\thanks{PREPRINT HAL-00429963 (2009)} 

\author{Serguei Barannikov}



\address{IMJ-PRG} 

\email{serguei.barannikov@imj-prg.fr}
\maketitle
\begin{abstract}
I describe a combinatorial construction of the cohomology classes
in compactified moduli spaces of curves $\widehat{Z}_{I}\in H^{*}(\bar{\mathcal{M}}_{g,n})$
starting from the following data: an\emph{ odd} derivation $I$, whose
square is \emph{non-zero} in general, $I^{2}\neq0$, acting on a $\mathbb{Z}/2\mathbb{Z}$-graded
associative algebra with \emph{odd} scalar product. The constructed
cocycles were first described in the theorem 2 in \cite{B3}. By the
theorems 2 and 3 from loc.cit. the family of the cohomology classes
obtained in the case of the algebra $Q(N)$ and the derivation $I=\left[\Lambda,\cdot\right]$
coincided with the generating function of products of $\psi$ classes.
This was the first nontrivial computation for \emph{categorical }Gromov-Witten
invariants of \emph{higher} genus.  As a byproduct I obtain
a new combinatorial formula for products of $\psi$-classes,
$\psi_{i}=c_{1}(T_{p_{i}}^{*})$, in the cohomology $H^{*}(\bar{\mathcal{M}}_{g,n})$
\end{abstract}

\section*{Introduction}

We have introduced in \cite{B3} a new type of graded matrix integrals
which can be seen as a generalization of the periods of projective
manifolds. One of their main properties was that their asymptotic
expansions defined cohomology classes of compactified moduli spaces
of curves. This paper is about this combinatorial construction of
the cohomology classes of compactified moduli spaces of curves starting
from the following data: $\mathbb{Z}/2\mathbb{Z}$-graded associative
algebra $A$ equipped with \emph{odd} scalar product $g$ and an\emph{
odd} derivation $I$, whose square is \emph{non-zero} in general,
$I^{2}\neq0$, $\widehat{Z}_{I}\in H^{*}(\bar{\mathcal{M}}_{g,n})$. 

An example of such data is the ``odd matrix algebra'' $Q(N)$ equipped
with the odd (queer) trace and the derivation $I=\left[\Lambda,\cdot\right]$,
see section \ref{sec:odd-matrix-algebra}. 

The constructed cocycles were first described in the theorem 2 in
\cite{B3}. By the theorems 2 and 3 from loc.cit. the family of the
cohomology classes obtained in the case of the algebra $Q(N)$ and
the derivation $I=\left[\Lambda,\cdot\right]$ coincides with the
generating function of products of $\psi$ classes. This was the first
nontrivial computation for \emph{categorical }Gromov-Witten invariants
of higher genus. The result matched with the mirror symmetry prediction,
namely with the classical (non-categorical) Gromov-Witten descendent
invariants of a point for all genus. 

\subsection{A new formula for products of $\psi$-classes in $H^{*}(\bar{\mathcal{M}}_{g,n})$.}

As a byproduct of the calculation in the case of $Q(N)$ I get the
following new formula for products of $\psi$-classes, $\psi_{i}=c_{1}(T_{p_{i}}^{*})$,
in the cohomology $H^{*}(\bar{\mathcal{M}}_{g,n})$ 
\begin{multline}
\sum_{\sum d_{i}=d}\psi_{1}^{d_{1}}\ldots\psi_{n}^{d_{n}}\prod_{i=1}^{n}\frac{(2d_{i}-1)!!}{\lambda_{i}^{2d_{i}+1}}=\\
=\left[\sum_{\left[\widehat{G}\right]\in\Gamma_{g,n}^{dec,odd}}(\widehat{G},or(\widehat{G}))\frac{2^{-\chi(\widehat{G})}}{\left|\textrm{Aut}(\widehat{G})\right|}\prod_{e\in\textrm{Edge}(\hat{G})}\frac{1}{\lambda_{i(e)}+\lambda_{j(e)}}\right]\label{eq:psiprod}
\end{multline}
where the sum on the right is over \emph{stable ribbon} graphs of
genus $g$ with $n$ numbered punctures, with $2d+n$ edges, and such
that its vertices have cyclically ordered subsets of arbitrary \emph{odd}
cardinality, see theorem \ref{psi-identity}. In the simplest case,
corresponding to the top degree, the cohomology $H^{6g-6+2n}(\bar{\mathcal{M}}_{g,n})$
is 1-dimensional, the summation is over 3-valent ribbon graphs and
this formula then reproduces the main identity from \cite{K}.

\subsection{Overview of the construction.}

The construction associates the following weight $\widehat{Z}_{\widetilde{I}}^{\widehat{G},or(\widehat{G})}$
with any oriented \emph{stable ribbon} graph defined as contraction
of tensors:
\[
\widehat{Z}_{\widetilde{I}}^{\widehat{G},or(\widehat{G})}=\left\langle \bigotimes_{e\in Edge(\widehat{G})}(g_{\widetilde{I}}^{-1})_{e},\bigotimes_{v\in Vert(\widehat{G})}\alpha_{\sigma_{v},\gamma_{v}}\right\rangle 
\]
 The tensors $\alpha_{\sigma_{v},\gamma_{v}}$ associated with the
vertices of the stable ribbon graph are the compositions of the multiplication
tensor with insertions of tensors dual to the scalar product which
are defined using the cyclic orders on the subsets of flags $\sigma_{v}=(\rho_{1}\ldots\rho_{r})\ldots(\tau_{1}\ldots\tau_{t})$
of the vertex $v$ 
\begin{multline*}
\alpha_{\sigma_{v},\gamma_{v}}(\pi a_{\rho_{1}}\otimes\ldots\otimes\pi a_{\tau_{t}})=\\
g\left(\sum_{\mu}(-1)^{\epsilon}e^{\mu_{1}}a_{\rho_{1}}\ldots a_{\rho_{r}}e_{\mu_{1}}e^{\mu_{2}}\ldots e_{\mu_{b_{v}-1}}a_{\tau_{1}}\ldots a_{\tau_{t}},(\sum_{\xi,\zeta}(-1)^{\overline{e^{\xi}}\overline{e^{\zeta}}}e^{\xi}e^{\zeta}e_{\xi}e_{\zeta})^{\gamma_{v}}\right)
\end{multline*}
see definition \eqref{def:alphamulti}. The tensors associated with
the edges are 
\[
g^{-1}(\widetilde{I}^{\ast}\varphi,\psi)
\]
where $g^{-1}$ is the dual scalar product and $\widetilde{I}^{\ast}$
is dual to $\widetilde{I}$ which is the homotopy inverse to $I$.
The construction of such inverse is explained in propositions \ref{prop:IItilde1},
\ref{prop:nondegenquadr}. The theorem \ref{thm:boundaryZhatzero}
affirms that the sum over all equivalence classes of oriented stable
ribbon graphs 
\[
\widehat{Z}_{I}=\sum\widehat{Z}_{\widetilde{I}}^{\widehat{G},or(\widehat{G})}(\widehat{G},or(\widehat{G}))
\]
 is a cocycle in the stable ribbon graph complex. 

The construction is a corollary to the theorem 2 from \cite{B1} which
identifies the complex computing the cohomology of the compactified
moduli spaces of curves with the complex underlying the Feynman transform
of the modular operad of associative algebras with scalar product.
A $\mathbb{Z}/2\mathbb{Z}$-graded associative algebra which satisfy
$\textrm{Tr}\left(l_{a}\right)=0$ (see \eqref{trMa}) is an algebra
over the modular operad $\mathbb{S}[t]$ from loc.cit. and the induced
morphism $\mathcal{F}(End(A),I)\rightarrow\mathcal{F}(\mathbb{S}[t])$
defines the cohomology class whose construction is detailed in this
paper. This in particular explains the tensors associated with the
vertices and gives an overall check of the construction independently
of the concrete calculations from the sections 7-9. 

\subsection{A counterexample to the theorem 1.3 from \cite{K2}.}

An important remark is that the truncated cochain, where the sum is
taken over the usual ribbon graphs only, is not closed. The boundary
of such truncated cochain gets certain nontrivial contribution from
graphs with loops and it is nonzero in general. The contribution is
non-zero even if one restricts the initial data by imposing extra
condition of $I^{2}=0$. This leads to a counterexample to the theorem
1.3 from \cite{K2}, see section \ref{sec:boundaryYes}. It turned
out that in the case of the simplest $\mathbb{Z}/2\mathbb{Z}$-graded
noncommutative algebra $\left\langle 1,\xi\right\rangle /\xi^{2}=1$,$\bar{\xi}=1$,
the cochain from loc.cit. does not vanish on the boundary of the graph
$G$, which represents a surface of genus 1 with 2 punctures, see Fig.\ref{fig1}
and proposition \ref{counterex}.

\subsection{Cohomological field theories.}

Our construction of the cohomology classes in $H^{*}(\bar{\mathcal{M}}_{g,n})$
defines a cohomological field theory. We extended this construction
of cohomology classes in $H^{*}(\bar{\mathcal{M}}_{g,n})$ to the
case of $A_{\infty}-$algebras/categories with scalar product, and
it defined also a cohomological field theory. This extension and the
cohomological field theory, that it defined, were studied in \cite{B4}.
This cohomological field theory for the algebra $A=End(C)$ where $C$ is a generator of $D^bCoh(X)$, $X$ is the mirror dual Calabi-Yau manifold,
conjecturally coincides with the cohomological field theory defined by Gromov-Witten invariants of all genus.

\subsection*{Plan of the paper}

In section 1 the homotopy invertability of the operator $I$ is studied.
In sections 2-5 the truncated cochain, which is a sum over usual ribbon
graphs only, is studied and it is shown that the boundary of the truncated
cochain is nonzero. Our main cocycle $\widehat{Z}_{I}$ is defined
in the sections 6-7. In sections 8-9 it is shown that the boundary
of $\widehat{Z}_{I}$ is zero. In sections 10 it is shown that modifying
$\widetilde{I}$ by the term $[\widetilde{I},[I,X]]$ produces a boundary,
and the independence of the cohomology class of $\widehat{Z}_{I}$
on the choice of $\widetilde{I}$ is studied in section 11. The section
12 concerns with the calculations in the case of $A=Q(N)$ and the
proof of the formula \eqref{eq:psiprod}

\subsection*{Acknowledgements.}

The results of this paper were presented starting from 2006 at conferences
in Cambridge, Berkeley, Grenoble, Bonn, Miami, Vienna, Brno, Tokyo, Moscow,
Boston. I am thankful to organizers of these conferences for their
hospitality, for the opportunity to present these results to large
audience and to the participants for interesting questions and comments.

\subsection*{Notations. }

For an element $a$ from $\mathbb{Z}/2\mathbb{Z}$-graded vector space
$A$ let us denote by$\pi a$ $\in\Pi A$ the same element considered
with inversed parity. $char(k)=0$. The parity of $a$ is denoted
by $\overline{a}$. Throughout the paper $A$ denotes a finite-dimensional\footnote{we work with finite-dimensional algebras, however it is straightforward
to adapt our construction to infinite dimensional situations, for
example of dg-associative algebra $A$ with $\dim_{k}H^{*}(A)<\infty$,
provided that the propagator $g^{-1}$ with the corresponding properties
is defined.} $\mathbb{Z}/2\mathbb{Z}$-graded associative algebra over $k$, $\textrm{char}(k)=0$,
with \emph{odd} invariant scalar product
\begin{eqnarray*}
g(\cdot,\cdot) & : & S^{2}A\rightarrow\Pi k\\
g(ab,c) & = & g(a,bc)\\
g(a,b) & = & (-1)^{\overline{a}\overline{b}}g(b,a)
\end{eqnarray*}
If the algebra has an identity then such a scalar product comes from
an \emph{odd} trace
\[
l:A/[A,A]\rightarrow\Pi k,g(a,b)=l(ab)
\]
In general the linear functional $l$ is well defined via $l(ab)=g(a,b)$
on the image of the multiplication map $A^{\otimes2}\rightarrow A$.
Denote by $I$ an odd anti-self-adjoint operator $I$ , which is a
derivation with respect to the algebra structure. 

\section{Propagators in the presence of odd scalar product.}

In this section the algebra structure does not play any role, so $g$
is an odd non-degenerate symmetric pairing $g:V^{\otimes2}\rightarrow\Pi k$
on some $\mathbb{Z}/2\mathbb{Z}$-graded vector space $V$, and $I$
is any odd anti-self-adjoint operator on $V$. 
\begin{theorem}
\label{thm:odd-anti-selfadj}Any odd anti-self-adjoint operator $I$
over algebraically closed field $k$, $char\left(k\right)\neq2$,
\[
g(Ix,y)+(-1)^{\overline{x}}g(x,Iy)=0
\]
can be reduced by linear transformation preserving the odd scalar
product $g$ to a sum consisting of the following blocks:

$1_{k}$- $\mathbb{Z}/2\mathbb{Z}-$graded subspace spanned by $2k$
elements $v,Iv,\ldots,I^{2k-1}v$ with $\overline{v}=k\textrm{ mod }2$,
$I^{2k}v=0$, 
\[
g(I^{m}v,I^{m'}v)=\delta_{m+m',2k-1}(-1)^{m\overline{v}+\frac{1}{2}m(m+1)},
\]

$2_{k}$- $\mathbb{Z}/2\mathbb{Z}-$graded subspace spanned by $2k$
elements $v,u,Iv,Iu\ldots,$ $I^{k-1}v,I^{k-1}u$ with $\overline{v}+\bar{u}=k\textrm{ mod }2$,
$I^{k}u=I^{k}v=0$, $g(I^{m}v,I^{m'}v)=0$, $g(I^{m}u,I^{m'}u)=0$,
\[
g(I^{m}v,I^{m'}u)=\delta_{m+m',k-1}(-1)^{m\overline{v}+\frac{1}{2}m(m+1)},
\]

$3$ - $\mathbb{Z}/2\mathbb{Z}-$graded subspace on which $I$ is
invertible.
\end{theorem}
\begin{svmultproof}
See \cite{B5} and also \cite{T}. 
\end{svmultproof}

Let $I$ be an odd anti-selfadjoint operator $I$ in the canonical
form of the direct sum of blocks of types $1_{k}$, $2_{k}$ and $3$.
\begin{proposition}
\label{prop:IItilde1}There exists an odd operator $\tilde{I}$, such
that 
\begin{equation}
\left[I,\tilde{I}\right]=1\label{eq:IItild}
\end{equation}
 if and only if the decomposition of the operator $I$ into blocks
by the theorem \ref{thm:odd-anti-selfadj} has no blocks of type $2_{2n-1}$,
$n\in\mathbb{N}$. In this case the operator $\tilde{I}$ can be chosen
to be self-adjoint $g(\tilde{I}x,y)=(-1)^{\overline{x}}g(x,\tilde{I}y)$. 
\end{proposition}
\begin{svmultproof}
Let us put for the elements of blocks of type $1_{k}$:
\[
\tilde{I}(I^{2l+1}v)=I^{2l}v,\;\tilde{I}(I^{2l}v)=0,\;l=0,\ldots,k-1,
\]
for the elements of blocks of type $2_{2k}$:
\[
\tilde{I}(I^{2l-1}v)=I^{2l}v,\;\tilde{I}(I^{2l-1}u)=I^{2l}u,\;\tilde{I}(I^{2l}v)=\tilde{I}(I^{2l}u)=0,\;l=0,\ldots,k-1,
\]
define also $\tilde{I}$ on block of type $3$:
\[
\tilde{I}=\frac{1}{2}I^{-1}.
\]
Then $\left[\tilde{I},I\right]=1$ and also $g(\tilde{I}x,y)=(-1)^{\overline{x}}g(x,\tilde{I}y)$. 

Assume now that such operator $\tilde{I}$ exists on a sum $V=W\oplus W'$
of two $I-$invariant subspaces, where $W$ is a block of type $2_{2k-1}$.
Let $\tilde{I}_{1}$ be the block-diagonal component of $\tilde{I}$
acting on $W$, $\tilde{I}=\begin{pmatrix}\tilde{I}_{1} & \tilde{I}_{3}\\
\tilde{I}_{2} & \tilde{I}_{4}
\end{pmatrix}$. Then $\left[\tilde{I}_{1},I\right]=1$on $W$, so it is enough to
consider the case $V=W$. Since $W$ is again the sum of two $I-$invariant
subspaces, it is enough to check that the operator satisfying the
\eqref{eq:IItild} cannot exist on the $I-$invariant subspace spanned
by $v,Iv,\ldots,I^{2k-2}v$. Since $I\left(I^{2k-2}v\right)=0$, then
$\tilde{I}\left(I^{2k-2}v\right)=I^{2k-3}v$, and therefore $\tilde{I}\left(I^{2k-3}v\right)=a_{0}I^{2k-2}v$.
It follows that $\tilde{I}\left(I^{2k-4}v\right)=I^{2k-5}v-a_{0}I^{2k-3}v$
and therefore $\tilde{I}\left(I^{2k-5}v\right)=a_{0}I^{2k-4}v+a_{1}I^{2k-2}v$,
etc. $\tilde{I}\left(Iv\right)=a_{0}I^{2}v+\sum_{j\geq1}a_{j}I^{2(j+1)}v$.
It follows that $\left[I,\tilde{I}\right]v\in\textrm{Image}\left(I\right)$
and therefore $\left[I,\tilde{I}\right]v\neq v$.
\end{svmultproof}

Such an odd anti-self-adjoint operator $I$ satisfying \eqref{eq:IItild}
is called homotopy invertible. 

This notion extends the notion of ``contractible differential''
to the case when $I^{2}\neq0$.

\subsection{Jourdan blocks decomposition for the pair of quadratic forms.}

The odd non-degenerate symmetric pairing $g$ is uniquely determined
by the linear isomorphism $g:V_{0}^{\vee}\simeq\Pi V_{1}$. The components
of the anti-self-adjoint operator $I$ are then the linear maps $I_{10}\in\textrm{Hom}_{k}(V_{0},V_{0}^{\vee})$,
$I_{01}\in\textrm{Hom}_{k}(V_{0}^{\vee},V_{0})$, $I_{10}$ is an
anti-symmetric pairing on $V_{0}$ and $I_{01}$ is a symmetric pairing
on $V_{0}^{\vee}$. 
\begin{proposition}
The anti-self-adjoint operators $I$ on $V$ are in one-to-one correspondence
with pairs $\left(h,h'\right)$ where $h$ is an anti-symmetric quadratic
form on $V_{0}$ and $h'$ is a symmetric quadratic form on $V_{0}^{\lor}$.
The decomposition of $I$ from the theorem \ref{thm:odd-anti-selfadj}
is equivalent to an analog of the ``Jordan blocks'' decomposition
for such pairs.
\end{proposition}

\subsection{Nondegeneracy of $g(Ix,x)$ on $\left(\Pi V\right)_{0}$ $\Rightarrow$
$I$ is homotopy invertible.}

The odd anti-self-adjoint operator $I$, or the pair $\left(h,h'\right)$
where $h$ is an anti-symmetric quadratic form on $V_{0}$ and $h'$
is a symmetric quadratic form on $V_{0}^{\lor}$, is the same as an
even quadratic function on $\Pi V$. The correspondence is given by
\[
I\leftrightarrow g(Ix,x),\;x\in\Pi V.
\]

\begin{proposition}
\label{prop:nondegenquadr}Let the restriction of the quadratic function
$g(Ix,x)$ on $\left(\Pi V\right)_{0}$ is a nondegenerate quadratic
form. Then the decomposition of $I$ has no blocks of type $2_{2n-1}$
and the operator $\widetilde{I}$ defined in the proposition \ref{prop:IItilde1}
satisfies. 
\[
\left[I,\tilde{I}\right]=1.
\]
\end{proposition}
\begin{svmultproof}
Notice that $\bar{v}+\bar{u}=1$ for the block of type $2_{2n-1}$,
and therefore one of the elements $I^{k-1}v,I^{k-1}u$ from the kernel
of $I$ is always odd and then the quadratic function $g(Ix,x)$ on
$\left(\Pi V\right)_{0}$ is degenerate.
\end{svmultproof}

\section{The ribbon graph complex.}

Recall the definition of the (even) ribbon graph complex $(C_{\ast},d)$,
see (\cite{GK}) and references therein.
\begin{definition}
A ribbon graph $G$ (without legs) is a triple $(Flag(G),$ $\sigma,$
$\eta)$, where $Flag(G)$ is a finite set, whose elements are called
flags, $\sigma$ is a permutation from $Aut(Flag(G))$ with orbits
of length greater than two, and $\eta$ is a fixed-point free involution
acting on $Flag(G)$. 
\end{definition}
The vertices of the graph correspond to cycles of the permutation
$\sigma$. The set of vertices is denoted by $Vert(G)$. The subset
of $Flag(G)$ corresponding to vertex $v$ is denoted by $Flag(v)$.
The set $Flag(v)$ has natural cyclic order since it is a cycle of
$\sigma$. The cardinality of $Flag(v)$ is called the valence of
$v$ and is denoted $n(v)$. It is assumed that $n(v)\geq3$. The
edges of the graph are the pairs of flags forming a two-cycle of the
involution $\eta$. The set of edges is denoted $Edge(G)$. The subset
of edges which are loops is denoted by $Loop(G)$. The edges which
are not loops are called \char`\"{}regular\char`\"{} and I denote
the corresponding subset of $Edge(G)$ by $Edge_{r}(G)$. The two
element subset of $Flag(G)$ corresponding to an edge $e\in Edge(G)$
is denoted $flag(e)$. I denote by $S_{G}$ the Riemann surface associated
with the ribbon graph $G$ (see loc.cit. and references therein).
I denote by $|G|$ the one-dimensional CW-complex which is the geometric
realization of the underlying graph $G$.
\begin{definition}
The even ribbon graph complex is the graded vector space generated
by equivalence classes of pairs $(G,or(G))$, where $G$ is a connected
ribbon graph, $or(G)$ is an orientation on the vector space 
\begin{equation}
\bigotimes_{v\in Vert(G)}(k^{Flag(v)}\oplus k),\label{eq:orRib}
\end{equation}
 with the relation $(G,-or(G))=-(G,or(G))$ imposed, equipped with
the differential 
\[
D(G,or(G))=\sum_{\left[e\right]\in Edge_{r}(G)}(G/\{e\},\text{induced\thinspace}\,\text{orientation})
\]
\end{definition}
The sum is over edges of $G$ which are not loops. The fact that the orientation $\eqref{eq:orRib}$
is equivalent to the orientation, defined as the orientation on the
vector space $(k^{Edge(G)}\oplus k^{H_{1}(G)})$, follows from (\cite{GK},
4.14). A choice of orientation can be fixed by a choice of a flag
from $Flag(v)$ for all vertices $v$ having \emph{even} valency and
a choice of an order on the set of such vertices. Two different choices
are related by a set of cyclic permutations on $Flag(v)$ for vertices
$v$ having even number of flags and a permutation on the set of such
vertices. The two corresponding orientations differ by the sign equal
to the product of signs of all these permutations.

The induced orientation on $G/\{e\}$ can be easily worked out. In
the case when the edge $e$ connects two vertices $v_{1}$ and $v_{2}$
with even number of flags, assume, that the order on the set of vertices
with even number of flags is such that $v_{1}<v_{2}$ and no other
such vertices are between them, and that the flags $f_{i}\in Flag(v_{i})$,
$i=1,2$, fixing the orientation, form the edge $e$. Then the induced
orientation for $G/\{e\}$ is defined by placing the new vertex $v_{0}$,
obtained from $v_{1}$ and $v_{2}$ after shrinking of $e$, to the
place of $v_{1}$ and by choosing the element corresponding to $\sigma_{G}(f_{1})$
in $Flag(v_{0})$. In the case when the edge $e$ connects two vertices
$v_{1}$ and $v_{2}$ with odd number of flags, the induced orientation
for $G/\{e\}$ is defined by placing the new vertex $v_{0}$, to the
lowest level in the set of vertices with even number of flags, and
by choosing the flag corresponding to $\sigma_{G}(f)$ in $Flag(v_{0}),$where
$f$ is either of the two flags forming $e$.

Denote via $(C^{\ast},d)$ the complex dual to the ribbon graph complex.
The generators of both complexes corresponding to oriented ribbon
graphs can be identified when it does not seem to lead to a confusion. 

\section{The truncated partition function\label{sec:partfunc}.}

The dual scalar product $g^{-1}$ is defined on $\Pi\limfunc{Hom}(A,k)$,
the dual space with parity inversed. If $P$ denotes the isomorphism
$A\rightarrow\Pi\limfunc{Hom}(A,k)$ induced by $g$, then 
\[
g^{-1}(\varphi,\psi)=g(P^{-1}\varphi,P^{-1}\psi).
\]
Recall $\widetilde{I}$ denotes the homotopy inverse to $I$ odd self-adjoint
operator, see propositions \ref{prop:IItilde1}, \ref{prop:nondegenquadr}.
Denote by $\widetilde{I}^{\ast}$the dual to $\widetilde{I}$, odd
self-adjoint operator, acting on $\Pi\limfunc{Hom}(A,k)$:
\[
\widetilde{I}^{\ast}\varphi(a)=(-1)^{\overline{\varphi}}\varphi(\widetilde{I}a),\,
\]
then 
\[
\widetilde{I}^{\ast}=P\widetilde{I}P^{-1},
\]
since $\widetilde{I}$ is self-adjoint. The odd self-adjoint operator
$\widetilde{I}^{\ast}$ defines the even symmetric pairing $g_{\widetilde{I}}^{-1}$
on $\Pi\limfunc{Hom}(A,k)$ : 
\[
g_{\widetilde{I}}^{-1}(\varphi,\psi)=g^{-1}(\widetilde{I}^{\ast}\varphi,\psi),\,\,\,g_{\widetilde{I}}^{-1}\in(\Pi A)^{\otimes2}
\]

\begin{definition}
Given a ribbon graph $G$, the tensor product of the even symmetric
tensors $(g_{\widetilde{I}}^{-1})_{e}$, associated with every edge
$e$ of $G$, $(g_{\widetilde{I}}^{-1})_{e}\in(\Pi A)^{\otimes flag(e)}$,
defines the canonical element 
\[
g_{\widetilde{I},G}=\bigotimes_{e\in Edge(G)}(g_{\widetilde{I}}^{-1})_{e},~~g_{\widetilde{I},G}\in(\Pi A)^{\otimes Flag(G)}.
\]

Define the $(-1)^{n+1}$-cyclic tensors $\alpha_{n}\in$ $\limfunc{Hom}((\Pi A)^{\otimes n},k)$,
$n\geq2$
\[
\alpha_{n}(\pi a_{1},\ldots,\pi a_{n})=(-1)^{\Sigma_{i=1}^{n-1}\Sigma_{j=1}^{i}\overline{a}_{j}}g(a_{1}\cdot a_{2}\cdot\ldots\cdot a_{n-1},\,a_{n})
\]
\end{definition}
(a tensor from $V^{\otimes n}$ is $(-1)^{n+1}$-cyclically symmetric
if under the elementary cyclic shift it is multiplied by the Koszul
sign resulting from parities of elements of $V$ times the sign of
the cyclic permutation, which is $(-1)^{n+1}$). 

Notice that the space $Flag(v)$ is canonically oriented, thanks to
the cyclic order, for a vertex $v\in Vert(G)$ of odd valency, $n(v)\equiv1\func{mod}2$.
For a vertex $v\in Vert(G)$ of the valency $n(v)\equiv0\func{mod}2$,
let $or(v)$ denotes an orientation on the vector space $Flag(v)$.
Because of cyclic order on $Flag(v)$, $or(v)$ is fixed canonically
by a choice of an element $f\in Flag(v)$.
\begin{definition}
The canonical elements 
\[
\alpha_{v}\in\limfunc{Hom}((\Pi A)^{\otimes Flag(v)},k)
\]
are defined by the cyclically symmetric tensors $\alpha_{2n+1}$,
$n\in\mathbb{N}$ for all vertices of $\emph{odd}$ valence and by
the cyclically \emph{anti-symmetric} tensors $\alpha_{2n+2}$, $n\in\mathbb{N}$
for all vertices of \emph{even }valence equipped with a choice of
orientation $or(v)$ fixed canonically by a choice of an element $f\in Flag(v)$. 
\end{definition}
\begin{proposition}
Given an oriented ribbon graph $G$ with orientation $or(G)$, the
product over all vertices $v$ of the $(-1)^{n(v)+1}$-cyclically
symmetric tensors $\alpha_{v}$ defines the canonical element 
\[
\alpha_{G,or(G)}=\bigotimes_{v\in Vert(G)}\alpha_{v},~~\alpha_{G,or(G)}\in\limfunc{Hom}((\Pi A)^{\otimes Flag(G)},k).
\]
\end{proposition}
\begin{svmultproof}
The orientation $or(G)$ is a choice of an element of $Flag(v)$ for
all vertices $v$ of \emph{even} valence, and a choice of an order
on the total set of such vertices of \emph{even }valencies. The first
part gives a choice of orientation $or(v)$ for all vertices of even
valence, and therefore well-defined elements $\alpha_{v}$ for all
vertices. Notice that the parity of $\alpha_{n}$ is even for odd
$n$, and odd for even $n$. Therefore with the choice of order on
the set of vertices of even valency the product $\bigotimes_{v\in Vert(G)}\alpha_{v}$
gives a well-defined element from $Hom((\Pi A)^{\otimes Flag(G)},k)$.
In more details, the orientation $or(G)$ fixes the signs in the definition
of $\alpha_{G,or(G)}\in\limfunc{Hom}((\Pi A)^{\otimes Flag(G)},k)$
by dictating for every vertex of even valence $n(v)$ which element
$a_{i}$ is to be placed the first inside $\alpha_{n(v)}$, and then
in which order the tensors $\alpha_{n(v)}$ are to be multiplied,
other choices are irrelevant because the sign of cyclic permutation
is always positive for odd number of elements, and because $\alpha_{n}$
defines an even tensor for odd $n$. 
\end{svmultproof}

\begin{definition}
The partition function $Z_{\widetilde{I}}^{G,or(G)}$ of an oriented
ribbon graph $G$ is the contraction of $\bigotimes_{v\in Vert(G)}\alpha_{v}$
with $\bigotimes_{e\in Edge(G)}(g_{\widetilde{I}}^{-1})_{e}$
\[
Z_{\widetilde{I}}^{G,or(G)}=\left\langle \bigotimes_{v\in Vert(G)}\alpha_{v},\bigotimes_{e\in Edge(G)}(g_{\widetilde{I}}^{-1})_{e}\right\rangle 
\]
 The sum over all equivalence classes of connected ribbon graphs defines
the cochain $Z_{\widetilde{I}}$ 
\begin{equation}
Z_{\widetilde{I}}=\sum_{[G]}Z_{\widetilde{I}}^{(G,or(G))}(G,or(G)),~Z_{\widetilde{I}}\in(C^{\ast},d)\label{ZI}
\end{equation}
\end{definition}
Notice that the element $Z_{\widetilde{I}}^{G,or(G)}$ $(G,or(G))$
of $(C^{\ast},d)$ does not depend on the choice of $or(G)$, since
a change of orientation changes the signs of both $Z_{\widetilde{I}}^{G,or(G)}$
and $(G,or(G))$.

\section{Action of $I$ on the tensors $g_{\widetilde{I}}^{-1}$ and $\alpha_{n}$.}

I denote by $I^{\ast}$ the odd dual to $I$ operator. It is an odd
anti-self-adjoint operator and it is acting on $\limfunc{Hom}(\Pi A,k)$
via:
\[
I^{\ast}\varphi(a)=(-1)^{\overline{\varphi}}\varphi(Ia)\,
\]
The action of any endomorphism of the super vector space $\limfunc{Hom}(\Pi A,k)$
is naturally extended to its tensor algebra $\oplus_{n}(\limfunc{Hom}(\Pi A,k))^{\otimes n}$
as a derivation, by the Leibnitz rule. The dual action on $\oplus_{n}(\Pi A)^{\otimes n}$
corresponds to the similar extension of the action of $I$ on $\Pi A$.
\begin{proposition}
\label{IgI}The result of action of the operator $I$ on $g_{\widetilde{I}}^{-1}$is
\[
I(g_{\widetilde{I}}^{-1})=g^{-1}
\]
\end{proposition}
\begin{svmultproof}
\[
I(g_{\widetilde{I}}^{-1})=g_{\widetilde{I}}^{-1}(I^{\ast}\varphi,\psi)+(-1)^{\overline{\varphi}}g_{\widetilde{I}}^{-1}(\varphi,I^{\ast}\psi)=g^{-1}([I^{\ast},\widetilde{I}^{\ast}]\varphi,\psi)=g^{-1}(\varphi,\psi)
\]
\end{svmultproof}

Consider the partition functions of an oriented ribbon graph $G$
and the graph $G^{^{\prime}}=G/\{e^{\prime}\}$ obtained from $G$
by contracting the edge $e^{\prime}\in Edge_{r}(G)$, with its induced
orientation. The previous proposition shows that acting by $I$ on
the two-tensor associated with the edge $e^{\prime}$ gives 
\[
I(g_{\widetilde{I}}^{-1})_{e^{\prime}}=g_{e^{\prime}}^{-1},
\]
where $g_{e^{\prime}}^{-1}\in(\Pi A)^{\otimes\{f,f^{\prime}\}}$,
$e^{\prime}=(ff^{\prime})$, denotes the two-tensor inverse to $g$,
associated with the edge $e^{\prime}$. I claim that inserting $g_{e^{\prime}}^{-1}$
instead of $(g_{\widetilde{I}}^{-1})_{e^{\prime}}$ for the edge $e^{\prime}$
\emph{which is not a loop} gives $Z_{\widetilde{I}}^{(G/\{e^{\prime}\},or(G/\{e^{\prime}\})}$.
\begin{proposition}
\label{IPHI}The partition function $Z_{\widetilde{I}}^{G/\{e^{\prime}\},or(G/\{e^{\prime}\})}$,
for a regular edge $e^{\prime}$, is equal to the contraction of $\bigotimes_{v\in Vert(G)}\alpha_{v}$
with $\bigotimes_{e\in Edge(G)^{\prime}}h_{e}$ where $h_{e}=(g_{\widetilde{I}}^{-1})_{e}$
for $e\neq e^{\prime}$ and $h_{e^{\prime}}=g_{e^{\prime}}^{-1}$. 
\end{proposition}
\begin{svmultproof}
The part of the contraction involving $g_{e^{\prime}}^{-1}$ is
\[
\sum_{\mu,\nu}(-1)^{\overline{\nu}}\alpha_{k+1}(v_{1}\ldots v_{k}u_{\mu})g^{-1}(u^{\mu},u^{\nu})\alpha_{n-k+1}(u_{\nu}v_{k+1}\ldots v_{n})
\]
where $\{u_{\mu}\}$ is a basis in $(\Pi A)$ and $\{u^{\mu}\}$ is
the dual basis, $u_{\mu}$ and $u_{\nu}$ represent the two flags
of the edge $e^{\prime}$, and 
\[
v_{1}\otimes\ldots\otimes v_{k}\otimes v_{k+1}\ldots\otimes v_{n}\in(\Pi A)^{\otimes Flag(v)}
\]
represents the flags corresponding to the new vertex of $G/\{e^{\prime}\}$
to which the edge $e^{\prime}$shrinks. Using the linear algebra identity
\[
b=\sum_{\mu}u_{\mu}u^{\mu}(b)=\sum_{\mu\nu}(-1)^{\overline{\nu}}u_{\mu}g^{-1}(u^{\mu},u^{\nu})g(u_{\nu},b).
\]
with $b=v_{k+1}\cdot\ldots\cdot v_{n}$ I get 
\[
\alpha_{n}(v_{1}\ldots v_{n})=\sum_{\mu,\nu}(-1)^{\overline{\nu}}\alpha_{k+1}(v_{1}\ldots v_{k}u_{\mu})g^{-1}(u^{\mu},u^{\nu})\alpha_{n-k+1}(u_{\nu}v_{k+1}\ldots v_{n})
\]
i.e. the tensor $\alpha_{v}$ associated with the new vertex. 
\end{svmultproof}

Next proposition shows that acting by $I^{\ast}$ on $\alpha_{n}$
gives zero because $I$ is a derivation of $A$, preserving $g$.
\begin{proposition}
\label{Ialph}
\[
I^{\ast}(\alpha_{n})=0
\]
\end{proposition}
\begin{svmultproof}
It follows from 
\begin{equation}
\sum_{i=1}^{n}(-1)^{\varepsilon_{i}}l(a_{1}\cdot\ldots\cdot(Ia_{i})\cdot\ldots\cdot a_{n})=l(I(a_{1}\cdot\ldots\cdot a_{n}))=0\label{lia}
\end{equation}
where $\varepsilon_{i}=\sum_{j=1}^{i-1}\overline{a}_{j}$, $n\geq2$
and $l$ is the odd linear functional $l(ab)=g(a,b)$ defined on the
image of the multiplication map $A^{\otimes2}\rightarrow A$.
\end{svmultproof}

\section{The boundary of the cochain $Z_{\widetilde{I}}$.\label{sec:boundaryYes}}

Combining the three propositions we get the following theorem.
\begin{theorem}
The boundary of the cochain $Z_{\widetilde{I}}$ (\ref{ZI}) is given
by the sum over graphs with loops, 
\[
dZ_{\widetilde{I}}=\sum_{[G],Loop(G)\neq\varnothing}(G,or(G))Z_{\widetilde{I}}^{\text{loop},(G,or(G))}
\]
each such graph with loops contributing 
\[
Z_{\widetilde{I}}^{\text{loop},(G,or(G))}=-\left\langle (\bigotimes_{e\in Edge_{r}(G)}(g_{\widetilde{I}}^{-1})_{e})I(\bigotimes_{l\in Loop(G)}(g_{\widetilde{I}}^{-1})_{l}),(\bigotimes_{v\in Vert(G)}\alpha_{v})\right\rangle .
\]
\end{theorem}
\begin{svmultproof}
The boundary of $Z_{\widetilde{I}}$ 
\[
dZ_{\widetilde{I}}=\sum_{[G^{\prime}]}Z_{\widetilde{I}}^{(G^{\prime},or(G^{\prime}))}d(G^{\prime},or(G^{\prime}))=\sum_{[G]}(G,or(G))\sum_{e\in Edge_{r}(G)}Z_{\widetilde{I}}^{(G/\{e\},\text{ }or(G/\{e\}))}
\]
For any ribbon graph $G$ we have

\begin{multline*}
0\overset{(\text{Prop. }\ref{Ialph})}{=}\left(\bigotimes_{e\in Edge(G)}(g_{\widetilde{I}}^{-1})_{e}\right)I^{\ast}(\bigotimes_{v\in Vert(G)}\alpha_{v})=\\
=I\left(\bigotimes_{e\in Edge(G)}(g_{\widetilde{I}}^{-1})_{e}\right)(\bigotimes_{v\in Vert(G)}\alpha_{v})\overset{(\text{Prop. }\ref{IgI},\ref{IPHI})}{=}\\
=\left(\sum_{e\in Edge_{r}(G)}Z_{\widetilde{I}}^{(G/\{e\},\text{ }or(G/\{e\}))}\right)+Z_{\widetilde{I}}^{\text{loop},(G,or(G))}
\end{multline*}
\end{svmultproof}

\begin{corollary}
The cochain $\sum_{[G]}Z_{\widetilde{I}}^{(G,or(G))}(G,or(G))$ is
not, in general, closed under the differential of the ribbon graph
complex. 
\end{corollary}
This gives a counterexample to the theorem 1.3 from \cite{K2}. Note
that by Poincare duality the rational homology of $\mathcal{M}_{g,n}$
coincide with the cohomology of one-point compactification which are
computed by the cohomology of the ribbon graph complex.
\begin{proposition}
\label{counterex}Consider the simplest noncommutative $\mathbb{Z}/2\mathbb{Z}-$graded
algebra $\left\langle 1,\xi\right\rangle /\xi^{2}=1$,$\bar{\xi}=1$,
equipped with the odd scalar product $g(1,\xi)=1$, the derivation
$I(\xi)=1$, $I^{2}=0$, and the homotopy inverse $\tilde{I}(1)=\xi$.
Then for the graph $G$ with three vertices and five edges, one vertex
of valence 4 with 2 opposite flags forming a loop, and two more vertices
of valence 3, see Fig.\ref{fig1}, so that $G$ represents a surface
of genus 1 with 2 punctures, $Z_{\widetilde{I}}(\partial G)\neq0$.
\end{proposition}
\begin{figure}
\begin{tikzpicture} \draw (1,0) arc (0:313.5:1) ; \draw (1,0) arc (0:-43:1) ;  \draw (-150:1) ..controls +(0.5,1) and +(0.5,-1).. (-30:1) ;  
\draw[<-,densely dotted,>=stealth] (-27.5:0.89) arc(120:480:0.12) ; \draw[<-,densely dotted,>=stealth] (-150:1.12) arc(-150:210:0.12) ; \draw[<-,densely dotted,>=stealth] (91:0.89) arc(-90:270:0.12) ; \draw (90:1) ..controls +(0,0.5) and +(70:0.5).. (70:1.05) ; \draw (90:1) ..controls +(0,-0.3) and +(250:0.3).. (70:0.95) ;
\end{tikzpicture}  \caption{$G$: $Z_{\widetilde{I}}(\partial G)\neq0$}    \label{fig1}
\end{figure}
\begin{svmultproof}
The contraction of any of two edges connecting the two vertices of
valence three gives the graph with two vertices of valence 4 on which
$Z_{\widetilde{I}}$ is zero. The graph $G$ has the $\mathbb{Z}/2\mathbb{Z}$
symmetry preserving the orientation, and the contraction of any of
the two remaining edges leads to the same contribution to $\partial G$
given by the graph with two vertices of valence 3 and 5, on which
$Z_{\widetilde{I}}$ is nonzero. 
\end{svmultproof}

\section{Compactification and the \emph{stable} ribbon graph complex.\label{sec:Comp-and-stable-ribbon}}

The graph with loops give non-zero terms in the boundary of the cochain
$Z_{\widetilde{I}}$ (\ref{ZI}) since loops are not allowed to contract
in the ribbon graph complex. There is a more general complex of \emph{stable}
ribbon graph, where loops are allowed to contract. It was proven in
\cite{B1} that the stable ribbon graph complex is isomorphic to the
complex underlying the Feynman transform of certain twisted modular
operad of associative algebras with scalar product. A corollary to
this theorem is the construction of the \emph{closed} cochain $\widehat{Z}_{\widetilde{I}}$
described in the sections \ref{sec:Comp-and-stable-ribbon}-\ref{sec:boundary-of-Zhat}.
In particular the weights $\alpha_{\sigma_{v},\gamma_{v}}$ \eqref{laaa}
associated to vertices are given by the action of the corresponding
elements of the twisted modular operad. 
\begin{definition}
A \emph{stable} ribbon graph $\widehat{G}$ is a data $(Flag(\widehat{G}),\sigma,\eta,\lambda,\{\gamma_{v}\})$,
where $Flag(\widehat{G})$ is a finite set, whose elements are called
flags, $\lambda$ is a partition on the set $Flag(\widehat{G})$,
$\sigma$ is a permutation from $Aut(Flag(\widehat{G}))$, stabilizing
$\lambda$, $\eta$ is a fixed-point free involution acting on $Flag(\widehat{G})$
, and $\gamma_{v}\in$ $\mathbb{Z}_{\geq0}$ is a nonnegative integer
attached to every cluster $v$ of $\lambda$. The clusters of the
partition $\lambda$ are the vertices of the \emph{stable} ribbon
graph $\widehat{G}$. The edges are the orbits of the involution $\eta$.
In particular to any vertex $v$ corresponds the integer $\gamma_{v}$
and a permutation $\sigma_{v}$ on the subset of flags $Flag(v)$
from the cluster $v$, so that $\sigma=\Pi_{v}\sigma_{v}$. Let $Cycle(v)$
denotes the set of cycles of the permutation $\sigma_{v}$ attached
to a vertex $v\in Vert(\widehat{G})$. The stable ribbon graph must
satisfy the stability condition for any $v$
\[
2(2\gamma_{v}+|Cycle(v)|-2)+\left|Flag(v)\right|>0.
\]
\end{definition}
Denote $b_{v}=|Cycle(v)|$ .

Given a stable ribbon graph $\hat{G}$ and a metric on $\hat{G}$
one can construct by the standard procedure a punctured Riemann surface
$S_{\widehat{G}}$, which will have singularities in general. Namely
replace every edge by oriented open strip $[0,l]\times]-i\infty,+i\infty[$
and glue them for each cyclically ordered subset according to the
cyclic order. In this way several punctured Riemann surfaces are obtained.
Then the points on these surfaces corresponding to different cyclically
ordered subsets associated with the given vertex of $\hat{G}$ should
be identified for every vertex of the graph $\hat{G}$. The one-dimensional
CW-complex $|\hat{G}|$ is naturally realized as a subset of $S_{\hat{G}}$.
It is also natural to glue in at each vertex with $2\gamma_{v}+b_{v}-2>0$
a topological surface of genus $\gamma_{v}$, which replaces the singular
point, obtained by identification of the $b_{v}$ points, by $b_{v}$
double points. 

\subsection{Contraction of edges on the stable ribbon graphs.}

In order to describe the action of the differential on the stable
ribbon graph $\widehat{G}$ one needs to describe the result of the
contraction of an arbitrary edge $\widehat{G}/\{e\}$. The contraction
of edges in the stable ribbon graphs is described combinatorially
via compositions and contractions on permutations, see (\cite{B1}),
representing the corresponding geometric operations on $S_{\widehat{G}}$.
For the reader convenience it is rephrased here. For an edge ending
at two vertices with $\gamma_{v}=0$, $b_{v}=1$ this is the standard
contraction of edge on ribbon graphs.

Let the permutation $\sigma^{\widehat{G}}$ is identified with multi-cyclic
order on $Flag(\widehat{G})$, i.e. the collection of cyclic orders
on the orbits of$\sigma$. Then the result of the contraction of an
edge $e=(ff^{\prime})$, $f,f^{\prime}\in Flag(\widehat{G})$ is the
stable ribbon graph $\widehat{G}/\{e\}$ such that 
\begin{eqnarray*}
Flag(\widehat{G}/\{e\}) & = & Flag(\widehat{G})\backslash\{f,f^{\prime}\},\\
\sigma^{\widehat{G}/\{e\}} & = & \left((ff^{\prime})\circ\sigma^{\widehat{G}}\right)|_{Flag(\widehat{G})\backslash\{f,f^{\prime}\}},
\end{eqnarray*}
i.e. $\sigma^{\widehat{G}/\{e\}}$ is the multi-cyclic order induced
on $Flag(\widehat{G})\backslash\{f,f^{\prime}\}$ by the multi-cyclic
order on $Flag(\widehat{G})$ given by the product $(ff^{\prime})\circ\sigma^{\widehat{G}}$,
\[
\eta^{\widehat{G}/\{e\}}=\eta^{\widehat{G}}|_{Flag(\widehat{G})\backslash\{f,f^{\prime}\}}
\]
i.e. the involution $\eta^{\widehat{G}/\{e\}}$ is the restriction
of the involution $\eta^{\widehat{G}}$ on the subset $Flag(\widehat{G})\backslash\{f,f^{\prime}\}$.
If an edge $e=(ff^{\prime})$, $f,f^{\prime}\in Flag(\widehat{G})$
is not a loop, i.e. $f,f^{\prime}$are from two distinguished clusters
$v$ and $v^{\prime}$ of $\lambda$, then the clusters $v$ and $v^{\prime}$
collide to the new cluster having $\gamma=\gamma_{v}+\gamma_{v^{\prime}}$.
If an edge $e$ is a loop then no vertices must collide, the partition
$\lambda^{\widehat{G}/\{e\}}$ on $Flag(\widehat{G})\backslash\{f,f^{\prime}\}$
is induced from $\lambda^{\widehat{G}}$ , and if the flags $f,f^{\prime}$
are from the same cycle of $\sigma_{v}^{\widehat{G}}$ then $\gamma_{v}^{\widehat{G}/\{e\}}=\gamma_{v}^{\widehat{G}}$,
otherwise if they are from different cycles of $\sigma_{v}^{\widehat{G}}$
then $\gamma_{v}^{\widehat{G}/\{e\}}=\gamma_{v}^{\widehat{G}}$. Lastly
if $f,f^{\prime}$ are neighbors in a cycle of $\sigma_{v}^{\widehat{G}}$,
i.e. say $\sigma_{v}^{\widehat{G}}(f)=f^{\prime}$, then, by definition,
$\widehat{G}/\{e\}=\varnothing$ , so that such loop do not contribute
to the boundary operator of the stable ribbon graph complex. This
exception is dictated by the relation of the stable ribbon graphs
with the combinatorial compactification of the moduli spaces and leads
to the condition on $A$, see section \ref{secanomaly}.

\subsection{The differential on the stable ribbon graph complex.\label{secdiff}}
\begin{definition}
The stable ribbon graph complex is the graded vector space generated
by equivalence classes of pairs $(\widehat{G},or(\widehat{G}))$,
where $\widehat{G}$ is a stable ribbon graph, $or(\widehat{G})$
is an orientation on the vector space 
\[
\bigotimes_{v\in Vert(\widehat{G})}(k^{Flag(v)}\oplus k^{Cycle(v)})
\]
and the relation $(\widehat{G},-or(\widehat{G})))=-(\widehat{G},or(\widehat{G}))$
is imposed. The differential is
\[
D(\widehat{G},or(\widehat{G}))=\sum_{\left[e\right]\in Edge(\widehat{G})}(\widehat{G}/\{e\},\text{induced\thinspace}\,\text{orientation})
\]
\end{definition}
The sum is over all edges of $\hat{G}$. Using the multi-cyclic order on $Flag(\widehat{G})$
, a choice of orientation is fixed by a choice of a flag from every
cycle of\emph{\ }$\sigma$ of \emph{even }length and a choice of
order on the total set of such cycles. The induced orientation on
$\widehat{G}/\{e\}$ is analogous to the case of usual ribbon graphs,
see (\cite{B1}). 

The differential in the dual complex: 
\[
d(\widehat{G}',or(\widehat{G}'))=\sum_{\widehat{G},\left[e\right]\in Edge(\widehat{G}),\,\widehat{G}/\{e\}=\widehat{G}'}(\widehat{G},or(\widehat{G}))
\]
is the sum over all stable graphs $\widehat{G}$ from which $\widehat{G}'$
can be obtained by the contraction of an arbitrary edge, equipped
with the orientation $or(\widehat{G})$ inducing the given orientation
$or(\widehat{G}')$.

\section{Weights on stable ribbon graphs.}

Let us start by constructing the tensor $\alpha_{\sigma_{v},\gamma_{v}}\in\limfunc{Hom}((\Pi A)^{\otimes Flag(v)},k)$
for any given vertex $v\in Vert(\widehat{G})$ with permutation $\sigma_{v}\in Aut(Flag(v))$
and integer $\gamma_{v}$ attached, together with a choice of orientation
$or(v)$ on $k^{Flag(v)}\oplus k^{Cycle(v)}$. Let $\sigma_{v}=(\rho_{1}\ldots\rho_{r})\ldots(\tau_{1}\ldots\tau_{t})$
be a representation of $\sigma_{v}$ compatible with $or(v)$, in
the sense that the order $\rho_{1}<\ldots<\tau_{t}$ on flags together
with the order $(\rho_{1}\ldots\rho_{r})<\ldots<(\tau_{1}\ldots\tau_{t})$
on cycles, is compatible with the choice of orientation on $k^{Flag(v)}\oplus k^{Cycle(v)}$.
\begin{definition}
\label{def:alphamulti}Put 
\begin{multline}
\alpha_{\sigma_{v},\gamma_{v}}(\pi a_{\rho_{1}}\otimes\ldots\otimes\pi a_{\tau_{t}})=\\
g\left((\sum_{\mu_{_{1}},\ldots,\mu_{b_{v}-1}}(-1)^{\epsilon}e^{\mu_{1}}a_{\rho_{1}}\ldots a_{\rho_{r}}e_{\mu_{1}}e^{\mu_{2}}\ldots e_{\mu_{b_{v}-1}}a_{\tau_{1}}\ldots a_{\tau_{t}}),\prod_{i=1}^{\gamma_{v}}(\sum_{\xi,\zeta}(-1)^{\overline{e^{\xi}}\overline{e^{\zeta}}}e^{\xi}e^{\zeta}e_{\xi}e_{\zeta})\right)\label{laaa}
\end{multline}
where $\{e_{\mu}\}$,$\{e^{\mu}\}$ is a pair of dual bases in $A$,
$g(e^{\mu},e_{\nu})=\delta_{\nu}^{\mu}$ , $\epsilon$ is the Koszul
sign taking into account the passing of all $\pi$'s to the left and
then putting $a_{i}$ inside the expression $(-1)^{\Sigma\overline{e^{\mu_{i}}}}e^{\mu_{1}}e_{\mu_{1}}\ldots e^{\mu_{b_{v}-1}}e_{\mu_{b_{v}-1}}$
. 
\end{definition}
\begin{proposition}
For any $a,b\in A$ 
\begin{equation}
\sum_{\mu}(-1)^{\overline{e^{\mu}}(\overline{a}+\overline{b}+1)}e^{\mu}abe_{\mu}=(-1)^{\overline{a}\overline{b}}\sum_{\nu}(-1)^{\overline{e^{\nu}}(\overline{a}+\overline{b}+1)}e^{\nu}bae_{\nu}.\label{eabe}
\end{equation}
For any $a,b\in A$, 
\begin{equation}
\sum_{\mu}(-1)^{\overline{e^{\mu}}(\overline{a}+1)}e^{\mu}ae_{\mu}b=(-1)^{(\overline{a}+1)\overline{b}}\sum_{\nu}(-1)^{\overline{e^{\nu}}(\overline{a}+1)}be^{\nu}ae_{\nu}.\label{eaeb}
\end{equation}
\end{proposition}
\begin{proof}
If 
\[
be_{\mu}=\sum_{\nu}\beta_{\mu}^{\nu}e_{\nu}
\]
then 
\[
\beta_{\mu}^{\nu}=g(e^{\nu},be_{\mu})=g(e^{\nu}b,e_{\mu})
\]
so
\[
e^{\nu}b=\sum_{\mu}\beta_{\mu}^{\nu}e^{\mu}.
\]
Therefore, 
\begin{multline*}
\sum_{\mu}(-1)^{\overline{e^{\mu}}(\overline{a}+\overline{b}+1)}e^{\mu}abe_{\mu}=\sum_{\mu}(-1)^{\overline{e^{\mu}}(\overline{a}+\overline{b}+1)}e^{\mu}a\sum_{\nu}\beta_{\mu}^{\nu}e_{\nu}=\\
=(-1)^{(\overline{a}+\overline{b}+1)\overline{b}}\sum_{\mu,\nu}(-1)^{\overline{e^{\nu}}(\overline{a}+\overline{b}+1)}\beta_{\mu}^{\nu}e^{\mu}ae_{\nu}=(-1)^{\overline{a}\overline{b}}\sum_{\nu}(-1)^{\overline{e^{\nu}}(\overline{a}+\overline{b}+1)}e^{\nu}bae_{\nu}
\end{multline*}
The proof of (\ref{eaeb}) is analogous. 
\end{proof}

\begin{corollary}
It follows that $\sum_{\mu}(-1)^{\overline{e^{\mu}}(\overline{a}+1)}e^{\mu}ae_{\mu}$
and $\sum_{\xi,\zeta}(-1)^{\overline{e^{\xi}}\overline{e^{\zeta}}}e^{\xi}e^{\zeta}e_{\xi}e_{\zeta}$
are in the center of $A$. 
\end{corollary}
It follows from (\ref{eabe}) and (\ref{eaeb}) that for every cycle
$(\rho_{1}\ldots\rho_{r})$ of $\sigma_{v}$ the expression (\ref{laaa})
is $(-1)^{r+1}$-cyclically symmetric with respect to cyclic permutations
of $\pi a_{\rho_{1}}\ldots\pi a_{\rho_{r}}$. And that the expression
(\ref{laaa}) is invariant under the changing the order of cycles
in the representation of $\sigma$, up to the sign taking into account
that the total parity of $e^{\mu_{1}}a_{\rho_{1}}\ldots a_{\rho_{r}}e_{\mu_{1}}$
differs from the parity of $\pi a_{\rho_{1}}\ldots\pi a_{\rho_{r}}$
by $r+1\textrm{ mod }2$. Therefore, for a fixed choice of $or(v)$,
the expression (\ref{laaa}) gives a well-defined element $\alpha_{\sigma_{v},\gamma_{v}}\in\limfunc{Hom}((\Pi A)^{\otimes Flag(v)},k)$.
Notice that for a vertex with $\gamma_{v}=0$ , $b_{v}=1$ this coincides
the tensor $\alpha_{v}$from section \ref{sec:partfunc}.

A choice of orientation $or(\widehat{G})$ on stable ribbon graph
can be identified with a choice of orientation $or(v)$ on every vertex
plus a choice of order on the set of vertices with \emph{odd }number
of cycles of \emph{even} length. These are precisely the vertices
for which the tensor $\alpha_{\sigma_{v},\gamma_{v}}$ is odd. Therefore
for a choice of $or(\widehat{G})$ the product $\bigotimes_{v\in Vert(\widehat{G})}\alpha_{\sigma_{v},\gamma_{v}}$gives
a well-defined element from $Hom((\Pi A)^{\otimes Flag(\widehat{G})},k)$.
\begin{proposition}
Given a choice of the orientation $or(\widehat{G})$ on the stable
ribbon graph $\widehat{G}$ the element 
\[
\alpha_{\widehat{G},or(\widehat{G})}=\bigotimes_{v\in Vert(\widehat{G})}\alpha_{\sigma_{v},\gamma_{v}}
\]
gives well-defined linear functional $\alpha_{\widehat{G},or(\widehat{G})}\in Hom((\Pi A)^{\otimes Flag(\widehat{G})},k)$.
\end{proposition}
$\square$

For the propagator I use the bilinear form associated with the homotopy
inverse of the odd supersymmetry $I$ from the section \ref{sec:partfunc}.
Given a a stable ribbon graph $\hat{G}$, the tensor product of the
even symmetric tensors $(g_{\widetilde{I}}^{-1})_{e}\in(\Pi A)^{\otimes Flag(e)}$,
associated with every edge $e$ of $\widehat{G}$, defines the canonical
element 
\[
g_{\widetilde{I},\widehat{G}}=\bigotimes_{e\in Edge(\widehat{G})}(g_{\widetilde{I}}^{-1})_{e},~~g_{\widetilde{I},\widehat{G}}\in(\Pi A)^{\otimes Flag(\widehat{G})}.
\]

\begin{definition}
The partition function $\widehat{Z}_{\widetilde{I}}^{\widehat{G},or(\widehat{G})}$
of an oriented stable ribbon graph $\widehat{G}$ is the contraction
of $\bigotimes_{v\in Vert(\widehat{G})}\alpha_{\sigma_{v},\gamma_{v}}$
with $\bigotimes_{e\in Edge(\widehat{G})}(g_{\widetilde{I}}^{-1})_{e}$:
\begin{equation}
\widehat{Z}_{\widetilde{I}}^{\widehat{G},or(\widehat{G})}=\left\langle g_{\widetilde{I},\widehat{G}},\alpha_{\widehat{G},or(\widehat{G})}\right\rangle .\label{eq:ZGhat}
\end{equation}
The element $\widehat{Z}_{\widetilde{I}}^{\widehat{G},or(\widehat{G})}\cdot(\widehat{G},or(\widehat{G}))$
of $(\widehat{C}^{\ast},d)$ does not depend on the choice of $or(\widehat{G})$.
The sum over all equivalence classes of connected stable ribbon graphs
defines the cochain $\widehat{Z}_{\widetilde{I}}$ 
\begin{equation}
\widehat{Z}_{\widetilde{I}}=\sum_{[G]}\widehat{Z}_{\widetilde{I}}^{\widehat{G},or(\widehat{G})}(\widehat{G},or(\widehat{G})),\;\widehat{Z}_{\widetilde{I}}\in(\widehat{C}^{\ast},d)\label{Zhat}
\end{equation}
\end{definition}

\section{Action of the derivation $I$ on the tensors $\alpha_{\widehat{G},or(\widehat{G})}$and
$g_{\widetilde{I},\widehat{G}}$.}

\subsection{Action of $I$ on the tensors $\alpha_{\sigma_{v},\gamma_{v}}$.}
\begin{proposition}
\begin{equation}
I^{\ast}(\alpha_{\sigma_{v},\gamma_{v}})=0\label{Ialphavg}
\end{equation}
\end{proposition}
\begin{svmultproof}
It follows from the fact that $I$ is a derivation of the multiplication
and of the odd scalar product . It follows from
\[
g(Ie^{\nu},e_{\mu})+(-1)^{\overline{e^{\nu}}}g(e^{\nu},Ie_{\mu})=0
\]
that if 
\[
Ie_{\mu}=\sum_{\nu}I_{\mu}^{\nu}e_{\nu}
\]
then 
\[
Ie^{\nu}=(-1)^{\overline{e^{\nu}}+1}\sum_{\mu}I_{\mu}^{\nu}e^{\mu}.
\]
Therefore 
\[
\sum_{\mu}I(e^{\mu})ae_{\mu}=\sum_{\mu,\nu}(-1)^{\overline{e^{\mu}}+1}I_{\nu}^{\mu}e^{\nu}ae_{\mu}=\sum_{\nu}(-1)^{\overline{e^{\nu}}}e^{\nu}a(Ie_{\nu})
\]
and 
\[
\sum_{\mu}(-1)^{\overline{e^{\mu}}(\overline{a}+1)}I(e^{\mu})ae_{\mu}+\sum_{\nu}(-1)^{\overline{e^{\nu}}(\overline{a}+1)+\bar{a}+\overline{e^{\nu}}}e^{\nu}aI(e_{\nu})=0
\]
for any $a\in A$. Now (\ref{Ialphavg}) follows from (\ref{lia}). 
\end{svmultproof}

The immediate consequence is the invariance of $\alpha_{\widehat{G},or(\widehat{G})}$
under the action of the derivation $I$.
\begin{proposition}
\label{IalphGhat}$I^{\ast}\alpha_{\widehat{G},or(\widehat{G})}=0$ 
\end{proposition}
$\square$

\subsection{The cancellation of anomaly.\label{secanomaly}}

From now on assume that for any $a\in$ $A$ the super trace of the
operator of (left) multiplication by $a$ is zero:
\begin{equation}
\sum_{\mu}(-1)^{\overline{\mu}}g(u^{\mu},a\cdot u_{\mu})=0\label{trMa}
\end{equation}
The supertrace of this operator is trivially zero for any odd $a$,
so it is sufficient to consider this condition for even $a$ only.
Let us denote by $K_{a}$ the operator of left multiplication by $a$
acting on $A$ and by $R_{a}$ the operator of right multiplication.
From 
\[
g(u^{\mu},a\cdot u_{\mu})=g(u^{\mu}\cdot a,u_{\mu})
\]
it follows that 
\[
\limfunc{Tr}K_{a}|_{A_{1}}=\limfunc{Tr}R_{a}|_{A_{0}}
\]
and therefore $\textrm{Tr}\left(K_{a}\right)=-\textrm{Tr}\left(R_{a}\right)$
\begin{proposition}
The following conditions are equivalent: $\textrm{Tr}\left(K_{a}\right)=0$,
$\textrm{Tr}\left(R_{a}\right)=0$, $\textrm{Tr}\left(\textrm{ad}_{a}\right)=0$.
\end{proposition}
$\square$
\begin{proposition}
Contracting any loop whose flags are neighbors, i.e. that $\sigma(f)=f^{\prime}$,
so that the loop encircles some boundary component of the surface
$S(\widehat{G})$, gives zero.
\[
\sum_{\mu\nu}(-1)^{\overline{\nu}}\alpha_{k+2}(v_{1}\ldots v_{k}u_{\mu}u_{\nu})g^{-1}(u^{\mu},u^{\nu})=0
\]
\end{proposition}
\begin{svmultproof}
Taking $a=v_{1}\cdot\ldots\cdot v_{k}$ the statement is reduced to
(\ref{trMa}). 
\end{svmultproof}

\begin{proposition}
Algebra $A$ with the odd scalar product $g$, which satisfies (\ref{trMa})
is the algebra over the modular operad $k\left[\mathbb{S}_{n}\right]\left[t\right]$
introduced in (\cite{B1}). The tensor $\alpha_{\sigma_{v},\gamma_{v}}$
is the result of the action of the corresponding element $\sigma_{v}$,
$\gamma_{v}$.
\end{proposition}
$\square$

\subsection{Action of $I$ on $g_{\widetilde{I},\widehat{G}}^{-1}$ and the contraction
of edges.}

As it was shown in the proposition \ref{IgI}, acting by $I$ on the
two-tensor associated with an edge $e^{\prime}$ gives 
\[
I(g_{\widetilde{I}}^{-1})_{e^{\prime}}=g_{e^{\prime}}^{-1},
\]
where $g_{e^{\prime}}^{-1}\in(\Pi A)^{\otimes\{f,f^{\prime}\}}$,
$e^{\prime}=(ff^{\prime})$, denotes the two-tensor inverse to $g$.
The next proposition shows that, inserting $g_{e^{\prime}}^{-1}$
instead of $(g_{\widetilde{I}}^{-1})_{e^{\prime}}$ for \emph{arbitrary}
edge $e^{\prime}$ and contracting with $\alpha_{\widehat{G},or(\widehat{G})}$
gives $\widehat{Z}_{\widetilde{I}}^{(\widehat{G}/\{e^{\prime}\},or(\widehat{G}/\{e^{\prime}\})}$,
the partition function of the oriented stable ribbon graph $\widehat{G}/\{e^{\prime}\}$.
\begin{proposition}
\label{ZhatGe}For \emph{arbitrary} edge $e^{\prime}$ the partition
function $\widehat{Z}_{\widetilde{I}}^{(\widehat{G}/\{e^{\prime}\},or(\widehat{G}/\{e^{\prime}\})}$
is equal to the contraction of $\alpha_{\widehat{G},or(\widehat{G})}$
with $\bigotimes_{e\in Edge(\widehat{G})^{\prime}}h_{e}$ where $h_{e}=(g_{\widetilde{I}}^{-1})_{e}$
for $e\neq e^{\prime}$ and $h_{e^{\prime}}=g_{e^{\prime}}^{-1}$. 
\end{proposition}
\begin{svmultproof}
Consider first the case when $e^{\prime}$ is not a loop. Then the
term of the contraction of $\bigotimes_{v\in Vert(\widehat{G})}\alpha_{v}$
with $\bigotimes_{e\in Edge(\widehat{G})^{\prime}}h_{e}$ , involving
$g_{e^{\prime}}^{-1}$ is 
\begin{equation}
\sum_{\mu,\nu}(-1)^{\overline{\nu}}\alpha_{\sigma_{v},\gamma_{v}}(a_{1},\ldots a_{k},U_{\mu})g^{-1}(U^{\mu},U^{\nu})\alpha_{\sigma_{v^{\prime}},\gamma_{v^{\prime}}}(U_{\nu},a_{k+1},\ldots,a_{n})\label{eq:alphagalpha}
\end{equation}
where $\{U_{\mu}\}$ is a basis in $(\Pi A)$ and $\{U^{\mu}\}$ is
the dual basis in $Hom((\Pi A),k)$ , $U_{\mu}$ and $U_{\nu}$ represent
the two flags of the edge $e^{\prime}$. Using the identities (\ref{eabe})
and (\ref{eaeb}) the tensor $\alpha_{\sigma_{v^{\prime}},\gamma_{v^{\prime}}}(U_{\nu},a_{k+1},\ldots,a_{n})$
can be represented in the form:
\[
g\left(\sum_{\mu_{_{1}},\ldots,\mu_{b_{v}-1}}(-1)^{\epsilon}U_{\nu}a_{\tau_{2}}\ldots a_{\tau_{t}}e^{\mu_{1}}a_{\rho_{1}}\ldots a_{\rho_{r}}e_{\mu_{1}}e^{\mu_{2}}\ldots e_{\mu_{b_{v}-1}},\prod_{i=1}^{\gamma_{v}}(\sum_{\xi_{i},\zeta_{i}}e^{\xi_{i}}e^{\zeta_{i}}e_{\xi_{i}}e_{\zeta_{i}})\right)
\]
Using the linear algebra identity 
\[
b=\sum_{\mu}U_{\mu}U^{\mu}(b)=\sum_{\mu\nu}(-1)^{\overline{\nu}}U_{\mu}g^{-1}(U^{\mu},U^{\nu})g(U_{\nu},b).
\]
with 
\[
b=\sum_{\mu_{_{1}},\ldots,\mu_{b_{v}-1}}(-1)^{\epsilon}a_{\tau_{2}}\ldots a_{\tau_{t}}e^{\mu_{1}}a_{\rho_{1}}\ldots a_{\rho_{r}}e_{\mu_{1}}e^{\mu_{2}}\ldots e_{\mu_{b_{v}-1}}\prod_{i=1}^{\gamma_{v}}(\sum_{\xi_{i},\zeta_{i}}e^{\xi_{i}}e^{\zeta_{i}}e_{\xi_{i}}e_{\zeta_{i}})
\]
the tensor \eqref{eq:alphagalpha} is identifed with the tensor $\alpha_{\sigma_{v},\gamma_{v}}(a_{1},\ldots a_{k},b)$.
Using again (\ref{eabe}) and (\ref{eaeb}) to bring it to the standard
form( \ref{laaa}), this gives precisely the tensor associated with
the new vertex in $\widehat{G}/\{e^{\prime}\}$, i.e. $\sigma_{v^{new}}$
is the merger of the two permutations $\sigma_{v^{\prime}}$ and $\sigma_{v^{\prime}}$
at the flags of $e^{\prime}$, and $\gamma_{v^{new}}=\gamma_{v}+\gamma_{v^{\prime}}$.
Let now $e^{\prime}=(ff^{\prime})$ is a loop , $f$, $f^{\prime}\in flag(v)$.
Then if $f$ and $f^{\prime}$ are in the same cycle of $\sigma_{v}$,
and if $f$, $f^{\prime}$ are neighbors, then the insertion of $g_{e^{\prime}}^{-1}$instead
of $(g_{\widetilde{I}}^{-1})_{e^{\prime}}$ gives zero because of
the property (\ref{trMa}) imposed on $A$. If flags $f$, $f^{\prime}$
are not neighbors, $\sigma(f)\neq f^{\prime}$, $\sigma(f^{\prime})\neq f$,
then 
\begin{multline*}
\sum_{\mu,\nu,\kappa}(-1)^{\overline{\nu}}g^{-1}(U^{\mu},U^{\nu})e^{\kappa}a_{\rho_{1}}\ldots a_{\rho_{i}}U_{\mu}a_{\rho_{i+1}}\ldots a_{\rho_{j-1}}U_{\nu}a_{\rho_{j}}\ldots a_{\rho_{r}}e_{\kappa}=\\
=\sum_{\mu,\kappa}e^{\kappa}a_{\rho_{1}}\ldots a_{\rho_{i}}e^{\mu}a_{\rho_{i+1}}\ldots a_{\rho_{j-1}}e_{\mu}a_{\rho_{j}}\ldots a_{\rho_{r}}e_{\kappa}
\end{multline*}
Using (\ref{eaeb}) with $b=a_{\rho_{j}}\ldots a_{\rho_{r}}e_{\kappa}$
and $a=a_{\rho_{i+1}}\ldots a_{\rho_{j-1}}$ it is transformed to
\[
\sum_{\mu,\kappa}(-1)^{\varepsilon}e^{\kappa}a_{\rho_{1}}\ldots a_{\rho_{i}}a_{\rho_{j}}\ldots a_{\rho_{r}}e_{\kappa}e^{\mu}a_{\rho_{i+1}}\ldots a_{\rho_{j-1}}e_{\mu}
\]
And the tensor $\alpha_{\sigma_{v}^{new},\gamma_{v}}$ corresponds
to the permutation $\sigma_{v}^{new}$ obtained from $\sigma_{v}$
by dissecting one cycle into two at the flags $f$ and $f^{\prime}$.
Similarly if $f$and $f^{\prime}$ are in the two different cycles
then the result is the tensor $\alpha_{\sigma_{v}^{new},\gamma_{v}^{new}}$
corresponding to the permutation $\sigma_{v}^{new}$ obtained from
$\sigma_{v}$ by merging the two cycles at flags $f$ and $f^{\prime}$
and with $\gamma_{v}^{new}=\gamma_{v}+1$. In both cases one gets
precisely the tensor corresponding to the transformed vertex $v$
in $\widehat{G}/\{e^{\prime}\}$. 
\end{svmultproof}

\section{The boundary of the cochain $\widehat{Z_{\widetilde{I}}}$.\label{sec:boundary-of-Zhat}}

The following main theorem is obtained by combining the results from
the previous sections.
\begin{theorem}
\label{thm:boundaryZhatzero}The boundary of the cochain $\widehat{Z}_{\widetilde{I}}$
(\ref{Zhat}) is zero
\[
d\widehat{Z}_{\widetilde{I}}=0
\]
\end{theorem}
\begin{svmultproof}
The boundary of $\widehat{Z}_{\widetilde{I}}$ 
\[
d\widehat{Z}_{\widetilde{I}}=\sum_{[\widehat{G}^{\prime}]}\widehat{Z}_{\widetilde{I}}^{(\widehat{G}^{\prime},or(\widehat{G}^{\prime}))}d(\widehat{G}^{\prime},or(\widehat{G}^{\prime}))=\sum_{[\widehat{G}]}(\widehat{G},or(\widehat{G}))\sum_{e\in Edge(\widehat{G})}\widehat{Z}_{\widetilde{I}}^{(\widehat{G}/\{e\},\text{ }or(\widehat{G}/\{e\}))}
\]
For any stable oriented ribbon graph $\widehat{G}$:
\begin{multline*}
\left(\sum_{e\in Edge(\widehat{G})}\widehat{Z}_{\widetilde{I}}^{(\widehat{G}/\{e\},\text{ }or(\widehat{G}/\{e\}))}\right)\overset{(\text{Prop.\ref{ZhatGe} })}{=}\\
=\left\langle I\left(\bigotimes_{e\in Edge(\widehat{G})}(g_{\widetilde{I}}^{-1})_{e}\right),\,\bigotimes_{v\in Vert(\widehat{G})}\alpha_{\alpha_{\sigma_{v},\gamma_{v}}}\right\rangle =\\
=\left\langle \bigotimes_{e\in Edge(\widehat{G})}(g_{\widetilde{I}}^{-1})_{e},\,I^{\ast}\left(\bigotimes_{v\in Vert(\widehat{G})}\alpha_{\alpha_{\sigma_{v},\gamma_{v}}}\right)\right\rangle \overset{(\text{Prop. \ref{IalphGhat}})}{=}0
\end{multline*}
\end{svmultproof}

\section{Producing boundaries.}

We saw above how to construct the cocycles associated with a pair
of odd operators $I$ and $\widetilde{I}$. Let $X$ is an arbitrary
odd linear operator $X\in End(A)$, such that $[I,X]$ is anti-selfadjoint.
In this section coboundaries are constructed by considering an infinitesimal
modification of $\widetilde{I}$ by the term $[\widetilde{I},[I,X]]$. 

Define the cochain $W_{\widetilde{I},X}$ as the contraction of $g_{\widetilde{I},\widehat{G}}^{-1}$with
$L_{X}(\alpha_{\widehat{G},or(\widehat{G})})$, where $L_{X}$ is
the dual to the extension of the action of $X$ on $\Pi A^{Flag(\widehat{G})}$
via the Leibnitz rule,
\begin{equation}
W_{\widetilde{I},X}=\left\langle g_{\widetilde{I},\widehat{G}},L_{X}\alpha_{\widehat{G},or(\widehat{G})}\right\rangle \label{wix}
\end{equation}

\begin{proposition}
\label{dwix}Under the infinitesimal modification of $\widetilde{I}$
\[
\widetilde{I}_{\varepsilon}=\widetilde{I}+\varepsilon[\widetilde{I},[I,X]]
\]
the partition function is changed by the boundary of $W_{\widetilde{I},X}$
\[
\widehat{Z}_{\widetilde{I}_{\varepsilon}}=\widehat{Z}_{\widetilde{I}}+\varepsilon d(W_{\widetilde{I},X})
\]
\end{proposition}
\begin{svmultproof}
\[
\widehat{Z}_{\widetilde{I}_{\varepsilon}}=\widehat{Z}_{\widetilde{I}}+\varepsilon\left\langle \sum_{e^{\prime}\in Edge(\widehat{G})}(g_{[\widetilde{I},[I,X]]}^{-1})_{e^{\prime}}\bigotimes_{e\in Edge(\widehat{G}),e\neq e^{\prime}}(g_{\widetilde{I}}^{-1})_{e},\alpha_{\widehat{G},or(\widehat{G})}\right\rangle 
\]
Notice that
\begin{eqnarray*}
g_{[\widetilde{I},[I,X]]}^{-1}(\varphi,\psi) & = & g^{-1}([\widetilde{I},[I,X]]\varphi,\psi)=g^{-1}(\widetilde{I}([I,X]\varphi),\psi)-g^{-1}([I,X]\widetilde{I}\varphi,\psi)=\\
 & = & g^{-1}(\widetilde{I}([I,X]\varphi),\psi)+g^{-1}(\widetilde{I}\varphi,[I,X]\psi)=L_{[I,X]}g_{\widetilde{I}}^{-1}(\varphi,\psi)
\end{eqnarray*}
since $X$ is anti-selfadjoint. Therefore
\[
\sum_{e^{\prime}\in Edge(\widehat{G})}(g_{[\widetilde{I},[I,X]]}^{-1})_{e^{\prime}}\bigotimes_{e\in Edge(\widehat{G}),e\neq e^{\prime}}(g_{\widetilde{I}}^{-1})_{e}=L_{[I,X]}g_{\widetilde{I},\widehat{G}}=[L_{I},L_{X}]g_{\widetilde{I},\widehat{G}}
\]
The action of $L_{I}$ on $g_{\widetilde{I},\widehat{G}}$ corresponds
to taking the differential: 
\begin{multline*}
\left\langle L_{X}L_{I}g_{\widetilde{I},\widehat{G}},\alpha_{\widehat{G},or(\widehat{G})}\right\rangle =\left\langle L_{I}g_{\widetilde{I},\widehat{G}},L_{X}\alpha_{\widehat{G},or(\widehat{G})}\right\rangle =\\
=\left\langle \sum_{e^{\prime}\in Edge(\widehat{G})}(g_{[\widetilde{I},I]}^{-1})_{e^{\prime}}\bigotimes_{e\in Edge(\widehat{G}),e\neq e^{\prime}}(g_{\widetilde{I}}^{-1})_{e},L_{X}\alpha_{\widehat{G},or(\widehat{G})}\right\rangle =d(W_{\widetilde{I},X}).
\end{multline*}
And because of 
\[
L_{I}\alpha_{\widehat{G},or(\widehat{G})}=0
\]
the other term vanishes: 
\[
\left\langle L_{I}L_{X}g_{\widetilde{I},\widehat{G}},\alpha_{\widehat{G},or(\widehat{G})}\right\rangle =\left\langle L_{X}g_{\widetilde{I},\widehat{G}},L_{I}\alpha_{\widehat{G},or(\widehat{G})}\right\rangle =0
\]
\end{svmultproof}

\section{Infinitesimal independence on the choice of $\widetilde{I}$.}

Let $\widetilde{I}$ be a homotopy inverse to $I$, odd, self-adjoint
operator, such that $\widetilde{I}^{2}=0$. Consider an infinitesimal
deformation of such operator $\widetilde{I}$:
\[
\widetilde{I}_{\varepsilon}=\widetilde{I}+\varepsilon Y,~~\varepsilon^{2}=0
\]

\begin{proposition}
The partition function is changed under such deformation by the boundary
of $W_{\widetilde{I},X}$ 
\[
\widehat{Z}_{\widetilde{I}_{\varepsilon}}=\widehat{Z}_{\widetilde{I}}+\varepsilon d(W_{\widetilde{I},X})
\]
where $W_{\widetilde{I},X}$ is the cochain defined in (\ref{wix})
with $X=\widetilde{I}IY$. 
\end{proposition}
\begin{svmultproof}
Such deformations satisfy: 
\[
\lbrack I,Y]=[\widetilde{I},Y]=0
\]
 Then
\[
Y=[\widetilde{I},[I,X]]
\]
with $X=\widetilde{I}IY$, since $\left[\widetilde{I},I^{2}\right]=0$.
And $[I,X]$ is anti-selfadjoint since $Y$ is self-adjoint and commutes
with $\widetilde{I}$ and $I$. Now the claim follows from the proposition
\ref{dwix}. 
\end{svmultproof}

\section{The odd matrix algebra $Q(N)$ and the $\psi$-classes\label{sec:odd-matrix-algebra}.}

Let $A$ be the ``odd matrix algebra'' $Q(N)$, see \cite{B2},
\cite{BL} and references therein. It is the associative algebra$Q(N)=\left\{ X\mid X\in gl(N\mid N),\left[X,p\right]=0\right\} $,
where $p=\begin{pmatrix}0 & 1\\
-1 & 0
\end{pmatrix}$. $Q(N)=gl(N)\oplus\Pi gl(N)$ as $\mathbb{Z}/2\mathbb{Z}-$graded
vector space. The algebra $Q(N)$ has the odd trace $otr(X)=\frac{1}{2}str(pX)$.
Let $I=\left[\Lambda,\cdot\right]$, $\Lambda\in Q(N)_{1}$ be the
odd derivation of the algebra $Q(N)$. Let us take the choice of $\Lambda\in Q(N)_{1}$as
$\varLambda=\frac{1}{2}\sum_{i}\lambda_{i}\Pi E_{i}^{i}$. I define
the operator $\widetilde{I}$ as $\widetilde{I}(E_{j}^{i})=\frac{2}{\lambda_{i}+\lambda_{j}}\Pi E_{j}^{i}$,
$\widetilde{I}(\Pi E_{j}^{i})=0$. Let $\widehat{Z}_{\widetilde{I}}(\lambda_{i})=\sum_{[\hat{G}]}\widehat{Z}_{\widetilde{I}(\lambda_{i})}^{\widehat{G},or(\widehat{G})}(\widehat{G},or(\widehat{G}))$
is the cocycle \eqref{ZhatGe} associated with the data $(Q(N),otr,\widetilde{I}(\lambda_{i}))$.
Let $\Gamma_{g,n}^{dec,odd}$denotes the set of isomorphism classes
of oriented \emph{stable} ribbon graphs of the type $(g,n)$ with
punctures decorated by $\left\{ 1,\ldots,N\right\} $ and such that
every vertex has cyclically ordered subsets of flags of \emph{odd}
cardinality only. 
\begin{proposition}
\label{weigth-ZIli}The weight $\widehat{Z}_{\widetilde{I}}^{\widehat{G},or(\widehat{G})}(\lambda_{i})$
associated with the data 
\[
(Q(N),otr,\widetilde{I}(\lambda_{i})),
\]
 is the sum over decorations of punctures of the stable graph $\widehat{G}$
by $\left\{ 1,\ldots,N\right\} $ of terms
\begin{equation}
\frac{2^{-\chi(\widehat{G})}}{\left|\textrm{Aut}(\widehat{G}^{\textrm{dec}})\right|}\prod_{e\in\textrm{Edge}(\hat{G})}\frac{1}{\lambda_{i(e)}+\lambda_{j(e)}}\label{eq:weightlambdaij}
\end{equation}
for stable ribbon graphs, such that for every vertex the cyclically
ordered subsets of flags are of \emph{odd} cardinality only, and zero
for other graphs. 
\end{proposition}
\begin{svmultproof}
The non-zero components of the propagator 
\[
g_{\widetilde{I}}^{-1}=\sum_{i,j}\frac{2}{\lambda_{i}+\lambda_{j}}\Pi E_{j}^{i}\otimes\Pi E_{i}^{j}
\]
 are from the tensor square of the odd part $\Pi gl(N)\subset Q(N)$
of the algebra $Q(N)$. The products of $r$ odd elements from $Q(N)$
is odd/even for odd/even $r$. Since the parity of elements $e^{\mu}$and
$e_{\mu}$is opposite, the tensors $\alpha_{\sigma_{v},\gamma_{v}}$
for $\sigma_{v}=(\rho_{1}\ldots\rho_{r})\ldots(\tau_{1}\ldots\tau_{t})$,
restricted to the tensor powers of the odd part $\Pi gl(N)\subset Q(N)$,
are nonzero only if the cardinality of each cyclically ordered subset
is odd. Since otherwise, the product of the elements from the subset
is even and the two terms involving $e^{\mu}=\Pi E_{j}^{i}$ and $e^{\mu}=E_{j}^{i}$
cancel each other. 
\[
\alpha_{\sigma_{v},\gamma_{v}}=2^{2\gamma_{v}+b_{v}-1}otr(a_{\rho_{1}}\ldots a_{\rho_{2r'+1}})\ldots otr(a_{\tau_{1}}\ldots,a_{\tau_{2t'+1}}),\,a_{i}\in\Pi gl(N)
\]
 where $otr$ denotes the odd trace $Q(N)$.The simple vertices with
only one cyclically ordered subset must be of odd valency in order
to contribute nonzero tensor $\alpha_{\sigma_{v},\gamma_{v}}$, in
particular. It is then follows from the standard Feynman diagrammatics
with matrix indices that the contraction \eqref{eq:ZGhat} is the
sum over decorations of boundary components of $\widehat{G}$ by $\left\{ 1,\ldots,N\right\} $
of terms \eqref{eq:weightlambdaij}.
\end{svmultproof}

\begin{proposition}
For any $(g,n)$ the cochain given by the sum over isomorphism classes
of oriented stable ribbon graphs with punctures decorated by $\left\{ 1,\ldots,N\right\} $
and such that every vertex has the cyclically ordered subsets of flags
of \emph{odd} cardinality only, taken with the weight $\widehat{Z}_{\widetilde{I}}^{\widehat{G},or(\widehat{G})}(\lambda_{i})$,
is a cocycle of the stable ribbon graph complex
\[
d\left(\sum_{\left[\widehat{G}\right]\in\Gamma_{g,n}^{dec,odd}}(\widehat{G},or(\widehat{G}))\frac{2^{-\chi(\widehat{G})}}{\left|\textrm{Aut}(\widehat{G})\right|}\prod_{e\in\textrm{Edge}(\hat{G})}\frac{1}{\lambda_{i(e)}+\lambda_{j(e)}}\right)=0.
\]
\end{proposition}
$\square$

There are natural differential 2- forms $\omega_{i}$ on the orbi-cell complex of
stable ribbon graphs, representing the cohomology classes $\psi_{i}=c_{1}(T_{p_{i}}^{*}\mathcal{\bar{M}}_{g,n})$,
see \cite{K}.
\begin{proposition}
\label{prop:The-Laplace-transf}The Laplace transform of the integral $\int_{C_{\hat{G}}}(\sum_{i=1}^{n}p_{i}^{2}\omega_{i})^{d}\prod dp_{i}$  over
the orbicell $C_{\hat{G}}$ corresponding to the oriented stable ribbon
graph $\widehat{G}$ coincides with 
\[
\frac{2^{-\chi(\widehat{G})}}{\left|\textrm{Aut}(\widehat{G}^{\textrm{dec}})\right|}\prod_{e\in\textrm{Edge}(\hat{G})}\frac{1}{\lambda_{i(e)}+\lambda_{j(e)}}
\]
\end{proposition}
$\square$
\begin{proposition}
\label{prop:The-value-zero}The value of the cochain representing
powers of combinatorial $\psi-$classes on the graphs with at least
one vertex having \emph{a} cyclically ordered subset of even cardinality
is zero.
\end{proposition}
\begin{proof}
Let $l_{1},l_{2}\ldots,l_{2r}$ denote the edges corresponding to
the cyclically ordered subset of some vertex, taken in the cyclic
order. I claim that the 2-form $\varOmega=\sum_{i=1}^{n}p_{i}^{2}\omega_{i}$
vanishes on the vector field $u=\frac{\partial}{\partial l_{1}}-\frac{\partial}{\partial l_{2}}+\ldots-\frac{\partial}{\partial l_{2r}}$.
This vector-field is tangent to the fibers $p_{i}=const$, since $dp_{i}(u)=0$.
It follows that the integral of $\Omega^{d}\cdot\prod dp_{i}$, $2d+n=\left|Edge(G)\right|$
over the cell corresponding to a graph with at least one vertex having
\emph{a} cyclically ordered subset of even cardinality is zero. 
\end{proof}

\begin{theorem}
\label{psi-identity}The following identity holds in the total cohomology
$H^{*}(\bar{\mathcal{M}}_{g,n})$.
\begin{multline*}
\sum_{\sum d_{i}=d}\psi_{1}^{d_{1}}\ldots\psi_{n}^{d_{n}}\prod_{i=1}^{n}\frac{(2d_{i}-1)!!}{\lambda_{i}^{2d_{i}+1}}=\\
=\left[\sum_{\left[\widehat{G}\right]\in\Gamma_{g,n}^{dec,odd}}(\widehat{G},or(\widehat{G}))\frac{2^{-\chi(\widehat{G})}}{\left|\textrm{Aut}(\widehat{G})\right|}\prod_{e\in\textrm{Edge}(\hat{G})}\frac{1}{\lambda_{i(e)}+\lambda_{j(e)}}\right]
\end{multline*}
where the sum on the right is over \emph{stable ribbon} graphs of
genus $g$ with $n$ numbered punctures, with $2d+n$ edges, and such
that its vertices have cyclically ordered subsets of arbitrary \emph{odd}
cardinality.
\end{theorem}
\begin{svmultproof}
The weight $\widehat{Z}_{\widetilde{I}}^{\widehat{G},or(\widehat{G})}(\lambda_{i})$
coincides with the Laplace transform of the integral $\int_{C_{\hat{G}}}(\sum_{i=1}^{n}p_{i}^{2}\omega_{i})^{d}\prod dp_{i}$
by the propositions \ref{weigth-ZIli}, \ref{prop:The-Laplace-transf}
and \ref{prop:The-value-zero}. Let $\varphi=\sum_{\widehat{G}}\varphi_{\widehat{G}}\left[\widehat{G}\right]$be
a cycle in the stable ribbon graph complex. For any $p_{i}>0$ the
integral of $\prod_{i=1}^{n}\varpi_{i}^{d_{i}}$ over $\varphi\cap\pi^{-1}(p_{i})$
does not depend on $p_{i}$ and coincides with the value of the cohomology
class of $\prod_{i=1}^{n}\varpi_{i}^{d_{i}}$ on the homology class
represented by $\varphi$, therefore the Laplace transform has poles
at $\lambda_{i}=0$ only and coincides with the generating function.
\end{svmultproof}

Let 
\[
\hat{F}(t_{0},t_{1},\ldots,t_{k},\ldots)=\sum_{n;d_{1},\ldots,d_{n}\geq0}\psi_{1}^{d_{1}}\ldots\psi_{n}^{d_{n}}\frac{t_{d_{1}}\ldots t_{d_{n}}}{n!}
\]
 be the power series with values in $\oplus_{g,n}H^{*}(\bar{\mathcal{M}}_{g,n})$
which is the generating function for the various products of the$\psi$-classes. 
\begin{theorem}
The cohomology class of $\widehat{Z}_{\widetilde{I}}(\lambda_{i})$
coincides with the formal power series $\hat{F}(t_{0}(\varLambda),t_{1}(\varLambda),\ldots,t_{k}(\varLambda),\ldots)$,
$t_{k}=-(2k-1)!!\sum_{i=1}^{n}\lambda_{i}^{-(2k+1)}$ with values
in the total cohomology $\oplus_{g,n}H^{*}(\bar{\mathcal{M}}_{g,n})$. 
\end{theorem}
\begin{svmultproof}
Follows from the theorem \ref{psi-identity}.
\end{svmultproof}

\end{document}